\documentclass{amsart} 
\usepackage{esint}
\usepackage{amsfonts}
\usepackage{amssymb}
\usepackage{xcolor}
\usepackage{url, hyperref}

\newtheorem{theorem}{Theorem}[section]

\newtheorem{corollary}[theorem]{Corollary}

\newtheorem{lemma}[theorem]{Lemma}

\newtheorem{proposition}[theorem]{Proposition}

\numberwithin{equation}{section}  
 
\begin{document}
\title{A Minkowski type inequality for manifolds with positive spectrum}
\author{Ovidiu Munteanu and Jiaping Wang}

\begin{abstract}
The classical Minkowski inequality implies that the volume of a bounded convex domain is controlled
from above by the integral of the mean curvature of its boundary. In this note,
we establish an analogous inequality without the convexity
assumption for all bounded smooth domains in a complete manifold with 
its bottom spectrum being suitably large relative to its Ricci curvature lower bound.
An immediate implication is the nonexistence of embedded compact minimal hypersurfaces in such manifolds.
This nonexistence issue is also considered for steady and expanding Ricci solitons.
\end{abstract}

\address{Department of Mathematics, University of Connecticut, Storrs, CT
06268, USA}
\email{ovidiu.munteanu@uconn.edu}
\address{School of Mathematics, University of Minnesota, Minneapolis, MN
55455, USA}
\email{jiaping@math.umn.edu}
\maketitle

\section{Introduction}

On a complete Riemannian manifold $(M, g),$ the Laplacian $\Delta$ is a self-adjoint operator according to \cite{Gaffney}. So the spectrum  
$\sigma(M)$ of $M,$ defined as the spectrum $\sigma(-\Delta)$ of $-\Delta,$ is a closed subset of $[0, \infty).$ The bottom spectrum is given by

$$
\lambda_1(M):=\min\{\lambda\in \sigma(M)\}.
$$ 
Alternatively, it is characterized as the best constant for the Poincar\'e inequality

$$
\lambda_1\,\int_M \phi^2\leq \int_M \left\vert \nabla \phi\right\vert^2
$$
for all compactly supported smooth functions $\phi$ on $M.$

A result by McKean \cite{McKean} says that $\lambda_1(M)\geq \frac{(n-1)^2}{4}$ for an $n$-dimensional, simply connected,
complete manifold $M^n$ with sectional curvature $K\leq -1$.

 The famous Sullivan-Patterson theory \cite{P, S} computes the bottom spectrum
for the quotient space $\mathbb{H}^n/\Gamma$ of the $n$-dimensional real hyperbolic space $\mathbb{H}^n,$ where
$\Gamma$ is a discrete, finitely generated, group of isometries of $\mathbb{H}^n.$ Namely,
$\lambda_1(\mathbb{H}^n/\Gamma)=\frac{(n-1)^2}{4}$ if $d_{\Gamma}\leq \frac{n-1}{2}$ and 
$\lambda_1(\mathbb{H}^n/\Gamma)=d_{\Gamma}\,(n-1-d_{\Gamma})$ when $d_{\Gamma}\geq \frac{n-1}{2},$ where $d_{\Gamma}$ is the Hausdorff
dimension of the limit set of $\Gamma,$ that is, those points $\theta$ in the ideal boundary at infinity $S_\infty(\mathbb{H}^n)$ of 
$\mathbb{H}^n$ such that $\theta=\lim_{i\to \infty} \gamma_i (x)$ for some $x\in \mathbb{H}^n$ and a sequence of 
$\gamma_i\in \Gamma$.

 Another notable result, due to Brooks \cite{Brooks}, is that $\lambda_1(M)>0$
for a covering space $M$ of a compact manifold $N$ if and only if the covering group is nonamenable.

Finally, we mention a result of Lee \cite{Lee}. Recall that a Riemannian manifold
$(M, g)$ is conformally compact if topologically it is the interior of a compact manifold $\overline{M}$ with boundary $N$ and its 
metric $g= \rho^{-2}\,g_{\overline{M}}$ 
for some metric $g_{\overline{M}}$ on $\overline{M}$ and smooth function $\rho$ on $\overline{M}$ with $\rho=0$ on $N$ and $d\rho\neq 0$ on $N.$ 
Note that different pairs of $\rho$ and $g_{\overline{M}}$ induce the same conformal class on $N.$ 

\begin{theorem}[Lee]
Let $(M^n, g)$ be a conformally compact Einstein manifold with its Ricci curvature normalized to be $-(n-1).$
If its boundary $N$ with the induced conformal metric has nonnegative scalar curvature, then $\lambda_1(M)=\frac{(n-1)^2}{4}.$ 
\end{theorem}

A different proof of the result was given by X. Wang \cite{Wang}.

Concerning the upper bound of the bottom spectrum, we have the following classical result due to Cheng \cite{C}.

\begin{theorem}[Cheng] 
Let $M^n$ be a complete Riemannian manifold with $\mathrm{Ric}\geq -(n-1)\kappa$ for some nonnegative constant $\kappa.$ Then
$$
\lambda_1(M)\leq \lambda_1\left(\mathbb{H}^n(-\kappa)\right)=\frac{(n-1)^2}{4}\,\kappa.
$$
\end{theorem}

The rigidity issue has been studied by Li and the second author in \cite{LW1, LW2}.

\begin{theorem}[Li-Wang] 
\label{LW}
Suppose $(M^n, g)$ is complete, $n\geq 3,$ with $\lambda_1\geq \frac{(n-1)^2}{4}\,\kappa$ 
and $\mathrm{Ric}\geq -(n-1)\kappa.$ Then either $M$ is connected at infinity or
$M^n=\mathbb{R}\times N^{n-1}$ for some compact $N$ with 
$g=dt^2+e^{2\sqrt{\kappa}\,t}\,g_N$ for $n\geq 3$ or $g=dt^2+\cosh^{2}(\sqrt{\kappa}\,t)\,g_N$ 
when $n=3.$
\end{theorem}

Note that as $\kappa$ goes to $0,$ the result recovers a weak version of the famous Cheeger-Gromoll \cite{CG}
splitting theorem for complete manifolds with nonnegative Ricci curvature.

Our main purpose here is to establish the following Minkowski type inequality for complete manifolds with positive bottom
spectrum.

\begin{theorem}\label{A_Intro}
Let $\left( M^{n},g\right) $ be a complete Riemannian manifold of dimension 
$n\geq 5$ with $\mathrm{Ric}\geq -\left( n-1\right) $ and 
\begin{equation*}
\lambda _{1}\left( M\right) \geq \left( \frac{n-2}{n-1}\right) ^{2}\left(
2n-3\right).
\end{equation*}
Then for any compact smooth domain $\Omega \subset M,$ 
\begin{equation*}
\frac{2}{3}\sqrt{n}\ \lambda _{1}\left( M\right) \mathrm{Vol}\left( \Omega
\right) \leq \int_{\partial \Omega }\left\vert H\right\vert ^{\frac{2n-3}{n-1}},
\end{equation*}
where $H$ is the mean curvature of $\partial \Omega.$
\end{theorem}
 
The result seems to be new even for the hyperbolic space $\mathbb{H}^n.$
We remark that it is necessary to assume $\lambda_1(M)>n-2.$ Indeed, for 
$M^n=\mathbb{R}\times N^{n-1}$ with 
$g=dt^2+\cosh^{2}(t)\,g_N,$ $\lambda_1(M)=n-2$ and $\mathrm{Ric}\geq -(n-1)$ when
$\mathrm{Ric_N}\geq -(n-2).$ Yet, the domain $\Omega$ given by $\{0<t<\varepsilon\}$ violates the inequality
when $\varepsilon$ is small. 

Certainly, this example also shows that the result can not hold for $n=3.$
However, it remains unclear what to expect for $n=4.$
One may wish to compare the result to the classical Minkowski inequality \cite{Min} for the
Euclidean space $\mathbb{R}^n$ and that for the hyperbolic space $\mathbb{H}^n$ \cite{GWW}. 
The advantage here is that no convexity is assumed for the domains.

\begin{theorem}[Minkowski]
If $\Omega\subset \mathbb {R}^n,$ $n\ge 3,$ is a convex domain with smooth boundary $\Sigma$ and $\mathrm{H}$ is 
the mean curvature of $\Sigma$ with respect to the outward unit normal, then there exists a sharp constant $c(n)$ so that

\begin{equation*}
\mathrm{Vol}\left( \Omega\right) \leq c(n)\,\left(\int_{\Sigma }H\right)^{\frac{n}{n-2}}.
\end{equation*}
Equality  holds if and only if $\Omega$ is a ball.  
\end{theorem}

The convexity can be relaxed to mean convex and star shaped by the work of Guan-Li \cite{GL1, GL2}, where
they produced a different proof using a new mean curvature flow. In fact, their proof 
yields the more general Alexandrov-Fenchel inequalities of quermassintegrals and extends to other space forms as well. 
For more related results, we refer to \cite{AMO2, BHW, CW, dG, GS}.

An immediate consequence of our result is the nonexistence of compact minimal hypersurfaces. 

\begin{corollary}
Let $\left( M^n, g\right)$  be a complete Riemannian manifold of dimension $n\geq 5$ 
with $\mathrm{Ric}\geq -\left( n-1\right) $ and 
\begin{equation*}
\lambda _{1}\left( M\right) =\frac{\left( n-1\right) ^{2}}{4}.
\end{equation*}
Then $M$ has no embedded compact minimal hypersurface. In particular, 
this holds for a conformally compact Einstein manifold with
its boundary having nonnegative scalar curvature.
\end{corollary}

Note that the result is not true for $n=3$. Indeed, for 
$M^3=\mathbb{R}\times N^2$ with 
$g=dt^2+\cosh^{2}(t)\,g_N,$ $\lambda_1(M)=1$ and $\mathrm{Ric}\geq -2$ when
$\mathrm{Ric_N}\geq -1.$ Yet, the hypersurface given by $\{t=0\}$ is totally geodesic.

The corollary follows from Theorem \ref{A_Intro} by verifying that $\Sigma$ must enclose a 
bounded domain $\Omega$ in $M.$ 
Indeed, observe that $M$ must be connected at infinity as otherwise by Theorem \ref{LW},

$$
M=\mathbb{R}\times N, \ \ ds_{M}^{2}=dt^{2}+e^{2 t}\,ds_{N}.
$$
Since $f(t, y)=t$ on $M$ is convex, by maximum principle, $\Sigma$ must be one of the level sets of $f.$ 
However, each level set has constant mean curvature $n-1$, which is a contradiction.
The same argument shows that every double cover of $M$ is connected at infinity as well. One then concludes
from a result by Carron and Pedon \cite{CP} that the integral homology $H_{n-1}(M, \mathbb{Z})=0.$
It then follows that $\Sigma$ must enclose a bounded domain $\Omega$ in $M.$ 

We now quickly sketch the proof of Theorem \ref {A_Intro}. First, there exists $v>0$ such that $\Delta v=-\lambda_1(M)\,v$. Consider the function $h=\ln v$, for which $\Delta h=-\lambda_1 (M) -\vert \nabla h\vert ^2$.   Then
 
\begin{eqnarray*}
\lambda _{1}\left( M\right) \mathrm{Vol}\left( \Omega\right)  &\leq
&\int_{\Omega}\left( \lambda _{1}\left( M\right) +\left\vert \nabla h\right\vert
^{2}\right)  \\
&=&-\int_{\Omega}\Delta h\\
& =& \int_{\partial \Omega }h_{\nu },
\end{eqnarray*}
where $\nu$ is the inward unit normal to $\partial \Omega$.

The proof is then reduced to estimating $\int_{\partial \Omega}h_{\nu }.$ To do so we
consider the harmonic function $u$ on $M\setminus \Omega$ obtained as $u=\lim_{R\to \infty}u_R,$ where
 $\Delta u_R=0$ on $\left(M\setminus \Omega\right)\cap B_p(R)$ with $u_R=1$ on $\partial \Omega$ 
 and $u_R=0$ on $\partial B_p(R).$ The upshot is to show 
 
\begin{equation*}
c(n)\,\int_{\partial \Omega}h_{\nu } \leq \int_{\partial \Omega}\left( \left\vert \nabla u\right\vert ^{\alpha
}\right) _{\nu }-\int_{\partial \Omega}\left( u^{\beta }\right) _{\nu}\left\vert \nabla u\right\vert ^{\alpha},
\end{equation*}
where $\alpha=\frac{n-2}{n-1}$ and $\beta=\frac{n-2}{3n-5}.$ 
For that, we drew inspiration from the monotonicity formulas for the Green's function 
on manifolds with nonnegative Ricci curvature \cite{Co, CM, AMO}, as well as
on $3$-dimensional manifolds with scalar
curvature bounded below \cite{MW}. In the process, the following generalized Poincar\'e inequality
also comes into play.

\begin{proposition}
Let $\left( M, g\right) $ be a complete manifold with $\mathrm{Ric}%
\geq -\left( n-1\right) $ and $\lambda _{1}\left( M\right) >0.$ Let $K\subset M$ 
be an open subset with (possibly noncompact) boundary $\partial K.$ 
Then the Poincar\'{e} inequality

\begin{equation*}
\lambda _{1}\left( M\right) \int_{K }\phi ^{2}\leq \int_{K}\left\vert \nabla \phi \right\vert ^{2}-
\int_{\partial K}h_{\nu} \,\phi ^{2}
\end{equation*}
holds for any Lipschitz function $\phi $ with compact support in $\overline{K},$ 
where $\nu $ is the outward unit normal to $\partial K.$ 
\end{proposition}

Concerning the nonexistence of compact minimal hypersurfaces, we also extend our consideration
to Ricci solitons. 

\begin{theorem}
\label{Steady0}Let $\left( M^{n},g,f\right) $ be a steady Ricci soliton.
If there exists a smooth compact embedded minimal hypersurface $\Sigma $ in $M,$
then $\left( M,g\right) $ splits isometrically as a direct product $\mathbb{R}\times \Sigma.$
\end{theorem}

A similar result is also
established for expanding Ricci solitons.
Recall that a gradient Ricci soliton is
a manifold $\left( M,g\right) $ such that there exists a smooth function $f$ satisfying

\begin{equation*}
\mathrm{Ric}_{f}=\mathrm{Ric}+\mathrm{Hess}\left( f\right)=\lambda g
\end{equation*}%
for some constant $\lambda \in \mathbb{R}.$ Solitons are classified as
shrinking, steady or expanding, according to $\lambda >0,\ \lambda =0$ or 
$\lambda <0,$ respectively. The function $f$ is called the potential.
Customarily, the constant $\lambda $ is assumed to be $1/2,$ $0,$ or $-1/2$
by scaling. Obviously, Ricci solitons are natural generalizations of
Einstein manifolds. More significantly, they are the self similar solutions to
the Ricci flows, and play a crucial role in the study of singularities of
the flows \cite{CLN}.

As pointed out in \cite{MW1}, an important feature of a steady Ricci soliton to us is that its 
bottom spectrum with respect to the weighted Laplacian $\Delta_f,$ 
defined by $\Delta_f u=\Delta u-\langle \nabla u, \nabla f\rangle,$ always achieves
the maximum value $\frac{1}{4}$ among all weighted manifolds $(M, g, e^{-f}dv)$ with $\mathrm{Ric}_{f}\geq 0$ and
$\left\vert \nabla f\right\vert\leq 1.$

The paper is arranged as follows. In Section 2, we prove our main result Theorem \ref{A_Intro}.
In Section 3, we consider Ricci solitons and prove Theorem \ref{Steady0}.

\section{Minkowski inequality}

In this section, we prove Theorem \ref{A_Intro}. First, we make some general consideration.
Throughout this section, unless otherwise noted, $\left( M^n, g\right) $ is assumed
to be an $n$-dimensional complete manifold with positive spectrum 
$\lambda_{1}\left( M\right) >0$ and its Ricci curvature $\mathrm{Ric}\geq -\left( n-1\right).$

It is well-known \cite{FS} that there exists $v>0$
such that 
\begin{equation}
\Delta v=-\lambda _{1}\left( M\right) v.  \label{x1}
\end{equation}%
Hence, 
\begin{equation}
h=\ln v  \label{h}
\end{equation}%
satisfies 
\begin{equation}
\Delta h=-\lambda _{1}\left( M\right) -\left\vert \nabla h\right\vert ^{2}.
\label{x2}
\end{equation}%
Also, by \cite{Yau} (or see Chapter 6 in \cite{L}), 
positive solutions of (\ref{x1}) satisfy the gradient estimate

\begin{equation}
\left\vert \nabla h\right\vert \leq \frac{n-1}{2}+\sqrt{\frac{\left(
n-1\right) ^{2}}{4}-\lambda _{1}\left( M\right) }.  \label{grad}
\end{equation}

The following generalized Poincar\'{e} inequality will be of use later.

\begin{proposition}
\label{P}Let $\left( M,g\right) $ be a complete manifold with $\mathrm{Ric}%
\geq -\left( n-1\right) $ and $\lambda _{1}\left( M\right) >0.$ Let $%
K\subset M$ be an open subset with boundary $\partial K$. Then 

\begin{equation*}
\lambda _{1}\left( M\right) \int_{K}\phi ^{2}\leq \int_{K}\left\vert \nabla
\phi \right\vert ^{2}-\int_{\partial K}\left\langle \nabla h,\nu
\right\rangle \phi ^{2}
\end{equation*}%
holds for any Lipschitz function $\phi $ with compact support in $\overline{K},$
where $\nu $ is the outward unit normal to the boundary $\partial K.$ 
In particular, 
\begin{equation*}
\lambda _{1}\left( M\right) \int_{K}\phi ^{2}\leq \int_{K}\left\vert \nabla
\phi \right\vert ^{2}+A\int_{\partial K}\phi ^{2},
\end{equation*}%
where 
\begin{equation*}
A=\frac{n-1}{2}+\sqrt{\frac{\left( n-1\right) ^{2}}{4}-\lambda _{1}\left(
M\right) }.
\end{equation*}
\end{proposition}

\begin{proof}
According to (\ref{x2}), for any Lipschitz function $\phi $ with compact support in $\overline{K}$ we have 
\begin{eqnarray*}
\lambda _{1}\left( M\right) \int_{K}\phi ^{2} &=&\int_{K}\left( -\Delta
h-\left\vert \nabla h\right\vert ^{2}\right) \phi ^{2} \\
&=&\int_{K}\left( \left\langle \nabla h,\nabla \phi ^{2}\right\rangle
-\left\vert \nabla h\right\vert ^{2}\phi ^{2}\right)  \\
&&-\int_{\partial K}h_{\nu }\phi ^{2}.
\end{eqnarray*}%
Observe that
\begin{equation*}
2\phi \left\langle \nabla h,\nabla \phi \right\rangle \leq \left\vert \nabla
h\right\vert ^{2}\phi ^{2}+\left\vert \nabla \phi \right\vert ^{2}.
\end{equation*}
Therefore,

\begin{equation*}
\lambda _{1}\left( M\right) \int_{K}\phi ^{2}\leq \int_{K}\left\vert \nabla
\phi \right\vert ^{2}-\int_{\partial K}\left\langle \nabla h,\nu
\right\rangle \phi ^{2}.
\end{equation*}%
This proves the result.
\end{proof}

We will apply this Poincar\'{e} inequality on the sublevel sets of the harmonic
function $u$ constructed below. Given a compact domain $\Omega \subset M,$ 
according to \cite{LW2}, an unbounded component of $M\setminus \Omega $ is 
parabolic if it has finite volume, and nonparabolic if it has infinite volume. 
Let $E_{1},..,E_{k}$ be all the infinite volume connected components of 
$M\setminus \Omega.$ Denote with $E=E_{1}\cup\cdots\cup E_{k}$ and

\begin{eqnarray}
D &=&M\setminus E  \label{D} \\
\Sigma  &=&\partial D=\partial E.  \notag
\end{eqnarray}%
Alternatively, $D$ is the union of $\Omega $ with all the finite volume components of 
$M\setminus \Omega.$

Consider the following function $u_i$ with respect to a sequence $R_i\to \infty.$

\begin{eqnarray}
\Delta u_{i} &=&0\text{ \ on }B_{p}\left( R_{i}\right) \setminus D
\label{ui} \\
u_{i} &=&1\text{ \ on }\partial D  \notag \\
u_{i} &=&0\text{ on }\partial B_{p}\left( R_{i}\right)\cap E .  \notag
\end{eqnarray}%
As $\lambda _{1}\left( M\right) >0,$ from \cite{LW1}, 
the sequence $\left\{ u_{i}\right\} _{i=1}^\infty$ converges to a positive
nonconstant harmonic function $u:M\setminus D\rightarrow \left[ 0,1\right] $
such that $u=1$ on $\partial D.$ The strong maximum principle implies that
$\left\vert \nabla u\right\vert >0$ on $\Sigma =\partial D$. Moreover, by 
\cite{LW2} 
\begin{equation}
\int_{M\setminus \left( D\cup B_{p}\left( R\right) \right) }u^{2}
\leq C\,e^{-2\sqrt{\lambda _{1}\left( M\right) }R}  \label{u}
\end{equation}
for all $R>0$ large enough. As $D$ is the union of $\Omega $ together
with all the finite volume components of $M\setminus \Omega,$ 
by \cite{LW2} the following volume estimate holds.

\begin{equation}
\mathrm{Vol}\left( D\setminus B_{p}\left( R\right) \right) 
\leq C\,e^{-2\sqrt{\lambda _{1}\left( M\right) }R}.  \label{VolFinite}
\end{equation}

We denote with 
\begin{eqnarray*}
L\left( \alpha ,\beta \right) &=&\left\{ x\in M\setminus D:\alpha <u\left(
x\right) <\beta \right\} \\
\ell\left( t\right) &=&\left\{ x\in M\setminus D:u\left( x\right) =t\right\} .
\end{eqnarray*}%
Note that these sets may be noncompact in general. However, (\ref{u})
implies that 
\begin{equation}
\mathrm{Vol}\left( L\left( \alpha ,\beta \right) \right) \leq \frac{1}{%
\alpha ^{2}}\int_{M\setminus D}u^{2}<\infty .  \label{V}
\end{equation}%
According to \cite{LW3}, 
\begin{equation}
\zeta =\int_{\ell\left( t\right) }\left\vert \nabla u\right\vert  \label{x4}
\end{equation}%
is a constant independent of $t\in \left[ 0,1\right] $. Hence, for any
function $F,$ by the co-area formula,
\begin{equation}
\int_{L\left( \alpha ,\beta \right) }\left\vert \nabla u\right\vert
^{2}F\left( u\right) =\zeta \int_{\alpha }^{\beta }F\left( t\right) dt.
\label{x5}
\end{equation}%
The gradient estimate for positive harmonic functions states that
\begin{equation}
\left\vert \nabla u\right\vert \leq C\,u\text{ \ on }M\setminus D,
\label{gradest}
\end{equation}%
where the constant $C$ depends only on the dimension $n$ and 
the maximum of the mean curvature $\max_{\Sigma }\left\vert H_{\Sigma }\right\vert.$

Recall the Bochner formula%
\begin{equation}
\frac{1}{2}\Delta \left\vert \nabla u\right\vert ^{2}\geq \left\vert \nabla
^{2}u\right\vert ^{2}-\left( n-1\right) \left\vert \nabla u\right\vert ^{2}
\label{BF}
\end{equation}%
and the improved Kato inequality 
\begin{equation}
\left\vert \nabla ^{2}u\right\vert ^{2}\geq \frac{n}{n-1}\left\vert \nabla
\left\vert \nabla u\right\vert \right\vert ^{2}.  \label{K}
\end{equation}

We begin with the following preliminary estimates.

\begin{lemma}
\label{Estimates}Let $\left( M^{n},g\right) $ be a complete manifold with $%
\mathrm{Ric}\geq -\left( n-1\right) $ and $\lambda _{1}\left( M\right) >0.$
There exists a constant $C>0$ such that for all $0<t<1,$
\begin{equation}
\int_{L\left( t,1\right) }\left( u+\left\vert \nabla \left\vert \nabla
u\right\vert \right\vert ^{2}\left\vert \nabla u\right\vert ^{-1}\right)
\leq C\left( 1-\ln t\right)   \label{e1}
\end{equation}
and 
\begin{equation}
\int_{L\left( \frac{1}{2}t,t\right) }\left( u+\left\vert \nabla \left\vert
\nabla u\right\vert \right\vert ^{2}\left\vert \nabla u\right\vert
^{-1}\right) \leq C. \label{e2}
\end{equation}
\end{lemma}

\begin{proof}
We first prove (\ref{e1}). Let $\psi $ and $\chi $ be the cut-off functions%
\begin{equation*}
\psi \left( x\right) =\left\{ 
\begin{array}{c}
1 \\ 
R+1-r\left( x\right)  \\ 
0%
\end{array}%
\right. 
\begin{array}{l}
\text{on }B_{p}\left( R\right)  \\ 
\text{on }B_{p}\left( R+1\right) \setminus B_{p}\left( R\right)  \\ 
\text{on }M\setminus B_{p}\left( R+1\right) 
\end{array}%
\end{equation*}%
and 
\begin{equation*}
\chi \left( x\right) =\left\{ 
\begin{array}{c}
1 \\ 
\frac{\ln u\left( x\right) -\ln \left( \frac{1}{2}t\right) }{\ln 2} \\ 
0%
\end{array}%
\right. 
\begin{array}{l}
\text{on }L\left( t,1\right)  \\ 
\text{on }L\left( \frac{1}{2}t,t\right)  \\ 
\text{otherwise}%
\end{array}%
.
\end{equation*}%
We extend $u=1$ on $D$, and let $\phi =u^{\frac{1}{2}}\chi \psi $ in 
\begin{equation*}
\lambda _{1}\left( M\right) \int_{M}\phi ^{2}\leq \int_{M}\left\vert \nabla
\phi \right\vert ^{2}
\end{equation*}%
to obtain 
\begin{eqnarray*}
\lambda _{1}\left( M\right) \int_{M}u\chi ^{2}\psi ^{2} &\leq
&2\int_{M}\left\vert \nabla u^{\frac{1}{2}}\right\vert ^{2}\chi ^{2}\psi
^{2}+2\int_{M}u\left\vert \nabla \left( \chi \psi \right) \right\vert ^{2} \\
&\leq &\frac{1}{2}\int_{M}\left\vert \nabla u\right\vert ^{2}u^{-1}\chi
^{2}\psi ^{2}+4\int_{M}u\left\vert \nabla \chi \right\vert ^{2}\psi ^{2} \\
&&+4\int_{M}u\left\vert \nabla \psi \right\vert ^{2}\chi ^{2}.
\end{eqnarray*}%
By (\ref{x5}) we immediately see that 
\begin{equation*}
\int_{M}\left\vert \nabla u\right\vert ^{2}u^{-1}\chi ^{2}\psi ^{2}\leq
\int_{L\left( \frac{1}{2}t,1\right) }\left\vert \nabla u\right\vert
^{2}u^{-1}=\zeta \ln \frac{2}{t}
\end{equation*}%
and that
\begin{equation}
\int_{M}u\left\vert \nabla \chi \right\vert ^{2}\psi ^{2}\leq \frac{1}{\left(\ln 2\right)^2}%
\int_{L\left( \frac{1}{2}t,t\right) }\left\vert \nabla u\right\vert
^{2}u^{-1}= \frac{\zeta}{\ln 2}.  \label{m0}
\end{equation}%
Finally, by (\ref{u}) and (\ref{VolFinite}) we have 
\begin{equation}
\int_{M}u\left\vert \nabla \psi \right\vert ^{2}\chi ^{2}\leq \frac{2}{t}%
\int_{M\setminus B_{p}\left( R\right) }u^{2}\leq \frac{C}{t}e^{-2\sqrt{%
\lambda _{1}\left( M\right) }R}.  \label{m1}
\end{equation}
This proves that 
\begin{equation*}
\int_{L\left( t,1\right) \cap B_{p}\left( R\right) }u\leq C\left( 1-\ln
t\right) +\frac{C}{t}e^{-2\sqrt{\lambda _{1}\left( M\right) }R}
\end{equation*}%
for all $R\geq 1.$ Making $R\rightarrow \infty $ implies that 
\begin{equation}
\int_{L\left( t,1\right) }u\leq C\left( 1-\ln t\right)   \label{m2}
\end{equation}
for all $0<t<1.$

By (\ref{BF}) and (\ref{K}) we have that 
\begin{equation*}
\Delta \left\vert \nabla u\right\vert \geq \frac{1}{n-1}\left\vert \nabla
\left\vert \nabla u\right\vert \right\vert ^{2}\left\vert \nabla
u\right\vert ^{-1}-\left( n-1\right) \left\vert \nabla u\right\vert
\end{equation*}%
on $M\setminus D$. It then follows that 
\begin{eqnarray}
\frac{1}{n-1}\int_{M\setminus D}\left\vert \nabla \left\vert \nabla
u\right\vert \right\vert ^{2}\left\vert \nabla u\right\vert ^{-1}\chi
^{2}\psi ^{2} &\leq &\int_{M\setminus D}\chi ^{2}\psi ^{2}\Delta \left\vert
\nabla u\right\vert  \label{m3} \\
&&+\left( n-1\right) \int_{M\setminus D}\left\vert \nabla u\right\vert \chi
^{2}\psi ^{2}.  \notag
\end{eqnarray}%
Note that by the gradient estimate (\ref{gradest}) and (\ref{m2}%
) we have 
\begin{eqnarray*}
\int_{M\setminus D}\left\vert \nabla u\right\vert \chi ^{2}\psi ^{2} &\leq
&C\int_{L\left( \frac{1}{2}t,1\right) }u \\
&\leq &C\left( 1-\ln t\right) .
\end{eqnarray*}
Moreover, integrating by parts implies that 
\begin{eqnarray*}
\int_{M\setminus D}\chi ^{2}\psi ^{2}\Delta \left\vert \nabla u\right\vert
&=&-\int_{M\setminus D}\left\langle \nabla \left( \chi ^{2}\psi ^{2}\right)
,\nabla \left\vert \nabla u\right\vert \right\rangle +\int_{\ell\left( 1\right)
}\left\vert \nabla u\right\vert _{\nu } \\
&\leq &\frac{1}{2\left( n-1\right) }\int_{M\setminus D}\left\vert \nabla
\left\vert \nabla u\right\vert \right\vert ^{2}\left\vert \nabla
u\right\vert ^{-1}\chi ^{2}\psi ^{2} \\
&&+2\left( n-1\right) \int_{M\setminus D}\left\vert \nabla u\right\vert
\left\vert \nabla \left( \chi \psi \right) \right\vert ^{2}+\int_{\ell\left(
1\right) }\left\vert \nabla u\right\vert _{\nu } \\
&\leq &\frac{1}{2\left( n-1\right) }\int_{M\setminus D}\left\vert \nabla
\left\vert \nabla u\right\vert \right\vert ^{2}\left\vert \nabla
u\right\vert ^{-1}\chi ^{2}\psi ^{2} \\
&&+C\left( 1-\ln t\right) +\frac{C}{t}e^{-2\sqrt{\lambda _{1}\left( M\right) 
}R},
\end{eqnarray*}%
where we have used (\ref{gradest}), (\ref{m0}) and (\ref{m1}) to obtain the last line.
Plugging this estimate in (\ref{m3}) and making $R\rightarrow \infty $ we
obtain that 
\begin{equation*}
\int_{L\left( t,1\right) }\left\vert \nabla \left\vert \nabla u\right\vert
\right\vert ^{2}\left\vert \nabla u\right\vert ^{-1}\leq C\left( 1-\ln
t\right)
\end{equation*}%
as claimed.

The second estimate (\ref{e2}) follows verbatim from the preceding argument
by modifying the function $\chi$ to
\begin{equation*}
\chi \left( x\right) =\left\{ 
\begin{array}{c}
1 \\ 
\frac{\ln u\left( x\right) -\ln \left( \frac{1}{4}t\right) }{\ln 2} \\ 
\frac{\ln \left( 2t\right) -\ln u}{\ln 2} \\ 
0%
\end{array}%
\right. 
\begin{array}{l}
\text{on }L\left( \frac{1}{2}t,t\right) \\ 
\text{on }L\left( \frac{1}{4}t,\frac{1}{2}t\right) \\ 
\text{on }L\left( t,2t\right) \\ 
\text{otherwise}%
\end{array}
.
\end{equation*}%
\end{proof}

We are ready to prove the main result of this section. 

\begin{theorem}
\label{Q}Let $\left( M^{n},g\right) $ be a complete Riemannian manifold of dimension $%
n\geq 5$ with $\mathrm{Ric}\geq -\left( n-1\right) $ and 
\begin{equation*}
\lambda _{1}\left( M\right) \geq \left( \frac{n-2}{n-1}\right) ^{2}\left(
2n-3\right) .
\end{equation*}%
Then for any compact smooth domain $\Omega \subset M,$ 

\begin{equation*}
\frac{2}{3}\sqrt{n}\ \lambda _{1}\left( M\right) \mathrm{Vol}\left( \Omega
\right) \leq \int_{\partial \Omega }\left\vert H\right\vert ^{\frac{2n-3}{n-1}},
\end{equation*}%
where $H$ is the mean curvature of $\partial \Omega.$
\end{theorem}

\begin{proof}
As in (\ref{D}) we let $D$ be the union of $\Omega $ with all the finite volume
components of $M\setminus \Omega.$ Define the harmonic function $u$ on $%
M\setminus D$ as the limit of a subsequence of $\left\{ u_{i}\right\} _{i=1}^\infty$
from (\ref{ui}).

Let $\psi $ and $\chi $ be the cut-off functions%
\begin{equation*}
\psi \left( x\right) =\left\{ 
\begin{array}{c}
1 \\ 
R+1-r\left( x\right)  \\ 
0%
\end{array}%
\right. 
\begin{array}{l}
\text{on }B_{p}\left( R\right)  \\ 
\text{on }B_{p}\left( R+1\right) \setminus B_{p}\left( R\right)  \\ 
\text{on }M\setminus B_{p}\left( R+1\right) 
\end{array}%
\end{equation*}%
and 
\begin{equation*}
\chi \left( x\right) =\left\{ 
\begin{array}{c}
1 \\ 
\frac{\ln u\left( x\right) -\ln \left( \frac{1}{2}t\right) }{\ln 2} \\ 
0%
\end{array}%
\right. 
\begin{array}{l}
\text{on }L\left( t,1\right)  \\ 
\text{on }L\left( \frac{1}{2}t,t\right)  \\ 
\text{otherwise}%
\end{array}%
.
\end{equation*}%
The Bochner formula (\ref{BF}) and the inequality (\ref{K}) imply the
inequality (cf. \cite{LW1}) 
\begin{equation}
\Delta \left\vert \nabla u\right\vert ^{\alpha }\geq -\left( n-2\right)
\left\vert \nabla u\right\vert ^{\alpha },  \label{x8}
\end{equation}%
where 
\begin{equation}
\alpha =\frac{n-2}{n-1}.  \label{alpha}
\end{equation}%
For 
\begin{equation}
\beta >\frac{1}{n-1}  \label{beta}
\end{equation}%
to be specified later, we multiply (\ref{x8}) by $u^{\beta }\chi ^{2}\psi
^{2}$ and integrate it over $M\setminus D$ to obtain that%
\begin{equation}
-\int_{M\setminus D}\left( \Delta \left\vert \nabla u\right\vert ^{\alpha
}\right) u^{\beta }\chi ^{2}\psi ^{2}\leq \left( n-2\right) \int_{M\setminus
D}\left\vert \nabla u\right\vert ^{\alpha }u^{\beta }\chi ^{2}\psi ^{2}.
\label{m4}
\end{equation}%
Integrating by parts one sees that the left side of (\ref{m4}) becomes

\begin{eqnarray*}
-\int_{M\setminus D}\left( \Delta \left\vert \nabla u\right\vert ^{\alpha
}\right) u^{\beta }\chi ^{2}\psi ^{2} &=&\int_{M\setminus D}\left\langle
\nabla \left\vert \nabla u\right\vert ^{\alpha },\nabla u^{\beta
}\right\rangle \chi ^{2}\psi ^{2} \\
&&+\int_{M\setminus D}\left\langle \nabla \left\vert \nabla u\right\vert
^{\alpha },\nabla \chi ^{2}\right\rangle u^{\beta }\psi ^{2} \\
&&+\int_{M\setminus D}\left\langle \nabla \left\vert \nabla u\right\vert
^{\alpha },\nabla \psi ^{2}\right\rangle u^{\beta }\chi ^{2} \\
&&-\int_{\ell\left( 1\right) }\left( \left\vert \nabla u\right\vert ^{\alpha
}\right) _{\nu },
\end{eqnarray*}%
where $\nu =\frac{\nabla u}{\left\vert \nabla u\right\vert }$ is the inward unit
normal to $\partial D=\ell\left( 1\right).$

By Lemma \ref{Estimates}, (\ref{gradest}),  and (\ref{u}) we have that 
\begin{eqnarray*}
\left\vert \int_{M\setminus D}\left\langle \nabla \left\vert \nabla
u\right\vert ^{\alpha },\nabla \psi ^{2}\right\rangle u^{\beta }\chi
^{2}\right\vert  &\leq &2\alpha \int_{L\left( \frac{1}{2}t,1\right)
\setminus B_{p}\left( R\right) }\left\vert \nabla \left\vert \nabla
u\right\vert \right\vert \left\vert \nabla u\right\vert ^{\alpha -1}u^{\beta
} \\
&\leq &Ce^{-\sqrt{\lambda _{1}\left( M\right) }R}\int_{L\left( \frac{1}{2}%
t,1\right) }\left\vert \nabla \left\vert \nabla u\right\vert \right\vert
^{2}\left\vert \nabla u\right\vert ^{-1} \\
&&+Ce^{\sqrt{\lambda _{1}\left( M\right) }R}\int_{L\left( \frac{1}{2}%
t,1\right) \setminus B_{p}\left( R\right) }u^{2\left( \alpha +\beta \right)
-1} \\
&\leq& Ce^{-\sqrt{\lambda _{1}\left( M\right) }R}\left( 1-\ln t\right)\\
&&+ \frac{C}{t} e^{\sqrt{\lambda _{1}\left( M\right) }R}\int_{L\left( \frac{1}{2}%
t,1\right) \setminus B_{p}\left( R\right) }u^2\\
&\leq &\frac{C}{t}e^{-\sqrt{\lambda _{1}\left( M\right) }R}.
\end{eqnarray*}%
Moreover, Lemma \ref{Estimates} also implies that 
\begin{eqnarray*}
\left\vert \int_{M\setminus D}\left\langle \nabla \left\vert \nabla
u\right\vert ^{\alpha },\nabla \chi ^{2}\right\rangle u^{\beta }\psi
^{2}\right\vert  &\leq &C\int_{L\left( \frac{1}{2}t,t\right) }\left\vert
\nabla \left\vert \nabla u\right\vert \right\vert \left\vert \nabla
u\right\vert ^{\alpha }u^{\beta -1} \\
&\leq &Ct^{\beta -\frac{1}{n-1}}\int_{L\left( \frac{1}{2}t,t\right)
}\left\vert \nabla \left\vert \nabla u\right\vert \right\vert ^{2}\left\vert
\nabla u\right\vert ^{-1} \\
&&+\frac{C}{t^{\beta -\frac{1}{n-1}}}\int_{L\left( \frac{1}{2}t,t\right)
}\left\vert \nabla u\right\vert ^{1+2\alpha }u^{2\beta -2} \\
&\leq &Ct^{\beta -\frac{1}{n-1}}+\frac{C}{t^{\beta -\frac{1}{n-1}}}%
\int_{L\left( \frac{1}{2}t,t\right) }\left\vert \nabla u\right\vert
^{2}u^{2\left( \alpha +\beta \right) -3},
\end{eqnarray*}%
where in the last line we have applied (\ref{gradest}). On the other hand, by (%
\ref{alpha}), (\ref{beta}), and (\ref{x5}) we get%
\begin{eqnarray*}
\int_{L\left( \frac{1}{2}t,t\right) }\left\vert \nabla u\right\vert
^{2}u^{2\left( \alpha +\beta \right) -3} &=&\zeta \int_{\frac{1}{2}%
t}^{t}r^{2\beta -\frac{n+1}{n-1}}dr \\
&\leq &\frac{1}{2\left( \beta -\frac{1}{n-1}\right) }t^{2\left( \beta -\frac{%
1}{n-1}\right) }\zeta .
\end{eqnarray*}%
In conclusion, 
\begin{equation*}
\left\vert \int_{M\setminus D}\left\langle \nabla \left\vert \nabla
u\right\vert ^{\alpha },\nabla \chi ^{2}\right\rangle u^{\beta }\psi
^{2}\right\vert \leq Ct^{\beta -\frac{1}{n-1}}.
\end{equation*}

Hence, this proves that 
\begin{eqnarray}
-\int_{M\setminus D}\left( \Delta \left\vert \nabla u\right\vert ^{\alpha
}\right) u^{\beta }\chi ^{2}\psi ^{2} &\geq &\int_{M\setminus D}\left\langle
\nabla \left\vert \nabla u\right\vert ^{\alpha },\nabla u^{\beta
}\right\rangle \chi ^{2}\psi ^{2}  \label{m5} \\
&&-\int_{\ell\left( 1\right) }\left( \left\vert \nabla u\right\vert ^{\alpha
}\right) _{\nu }  \notag \\
&&-\frac{C}{t}e^{-\sqrt{\lambda _{1}\left( M\right) }\,R}-Ct^{\beta -\frac{1}{%
n-1}}.  \notag
\end{eqnarray}

We now estimate the first term on the right hand side. Integration by parts implies that 
\begin{eqnarray*}
\int_{M\setminus D}\left\langle \nabla \left\vert \nabla u\right\vert
^{\alpha },\nabla u^{\beta }\right\rangle \chi ^{2}\psi ^{2}
&=&-\int_{M\setminus D}\left\vert \nabla u\right\vert ^{\alpha }\left(
\Delta u^{\beta }\right) \chi ^{2}\psi ^{2} \\
&&-\int_{M\setminus D}\left\vert \nabla u\right\vert ^{\alpha }\left\langle
\nabla u^{\beta },\nabla \left( \chi ^{2}\psi ^{2}\right) \right\rangle  \\
&&+\int_{\ell\left( 1\right) }\left( u^{\beta }\right) _{\nu }\left\vert \nabla
u\right\vert ^{\alpha }.
\end{eqnarray*}%
By (\ref{u}) and (\ref{gradest}) we have that 
\begin{eqnarray*}
\int_{M\setminus D}\left\vert \nabla u\right\vert ^{\alpha }\left\vert
\left\langle \nabla u^{\beta },\nabla \psi ^{2}\right\rangle \chi
^{2}\right\vert  &\leq &C\int_{L\left( \frac{1}{2}t,1\right) \setminus
B_{p}\left( R\right) }u^{\alpha +\beta }\chi ^{2} \\
&\leq &\frac{C}{t^{2-\left( \alpha +\beta \right) }}\int_{L\left( \frac{1}{2}%
t,1\right) \setminus B_{p}\left( R\right) }u^{2} \\
&\leq &\frac{C}{t}e^{-2\sqrt{\lambda _{1}\left( M\right) }R}.
\end{eqnarray*}%
Moreover, (\ref{x5}) and (\ref{gradest}) imply that%
\begin{eqnarray*}
\int_{M\setminus D}\left\vert \nabla u\right\vert ^{\alpha }\left\langle
\nabla u^{\beta },\nabla \chi ^{2}\right\rangle \psi ^{2} 
&\leq& C \int_{L\left( \frac{1}{2}t,t\right) }\left\vert \nabla u\right\vert
^{\alpha+2}u^{\beta -2}\\
&\leq&C\int_{L\left( \frac{1}{2}t,t\right) }\left\vert \nabla u\right\vert
^{2}u^{\beta -\frac{n}{n-1}} \\
&\leq &\frac{C}{\beta -\frac{1}{n-1}}t^{\left( \beta -\frac{1}{n-1}\right)
}\zeta .
\end{eqnarray*}%
Plugging these estimates in (\ref{m5}) yields that 
\begin{eqnarray}
-\int_{M\setminus D}\left( \Delta \left\vert \nabla u\right\vert ^{\alpha
}\right) u^{\beta }\chi ^{2}\psi ^{2} &\geq &\beta \left( 1-\beta \right)
\int_{M\setminus D}\left\vert \nabla u\right\vert ^{\alpha +2}u^{\beta
-2}\chi ^{2}\psi ^{2}  \label{m6} \\
&&-\int_{\ell\left( 1\right) }\left( \left\vert \nabla u\right\vert ^{\alpha
}\right) _{\nu }+\int_{\ell\left(1\right) }\left( u^{\beta }\right) _{\nu
}\left\vert \nabla u\right\vert ^{\alpha }  \notag \\
&&-\frac{C}{t}e^{-\sqrt{\lambda _{1}\left( M\right) }R}-Ct^{\beta -\frac{1}{%
n-1}}.  \notag
\end{eqnarray}

Therefore (\ref{m4}) becomes that
\begin{eqnarray}
\left( n-2\right) \int_{M\setminus D}\left\vert \nabla u\right\vert ^{\alpha
}u^{\beta }\chi ^{2}\psi ^{2} &\geq &\beta \left( 1-\beta \right)
\int_{M\setminus D}\left\vert \nabla u\right\vert ^{\alpha +2}u^{\beta
-2}\chi ^{2}\psi ^{2}  \label{m7} \\
&&-\int_{\ell\left( 1\right) }\left( \left\vert \nabla u\right\vert ^{\alpha
}\right) _{\nu }+\int_{\ell\left( 1\right) }\left( u^{\beta }\right) _{\nu
}\left\vert \nabla u\right\vert ^{\alpha }  \notag \\
&&-\frac{C}{t}e^{-\sqrt{\lambda _{1}\left( M\right) }R}-Ct^{\beta -\frac{1}{%
n-1}}.  \notag
\end{eqnarray}%
We now estimate the left hand side. By Young's inequality we have that  
\begin{equation*}
\left\vert \nabla u\right\vert ^{\alpha }\leq \frac{\alpha }{\alpha +2}%
A^{-2}\left\vert \nabla u\right\vert ^{\alpha +2}u^{-2}+\frac{2}{\alpha +2}%
A^{\alpha }u^{\alpha },
\end{equation*}%
where $A>0$ is a constant to be specified later. Hence, we obtain 
\begin{eqnarray*}
\int_{M\setminus D}\left\vert \nabla u\right\vert ^{\alpha }u^{\beta }\chi
^{2}\psi ^{2} &\leq &\frac{\alpha }{\alpha +2}A^{-2}\int_{M\setminus
D}\left\vert \nabla u\right\vert ^{\alpha +2}u^{\beta -2}\chi ^{2}\psi ^{2}
\\
&&+\frac{2}{\alpha +2}A^{\alpha }\int_{M\setminus D}u^{\alpha +\beta }\chi
^{2}\psi ^{2}.
\end{eqnarray*}%
Plugging this into (\ref{m7}) yields 
\begin{eqnarray}
\Lambda _{1}\int_{M\setminus D}\left\vert \nabla u\right\vert ^{\alpha
+2}u^{\beta -2}\chi ^{2}\psi ^{2} &\leq &\Lambda _{2}\int_{M\setminus
D}u^{\alpha +\beta }\chi ^{2}\psi ^{2}  \label{m8} \\
&&+\int_{\ell\left( 1\right) }\left( \left\vert \nabla u\right\vert ^{\alpha
}\right) _{\nu }-\int_{\ell\left( 1\right) }\left( u^{\beta }\right) _{\nu
}\left\vert \nabla u\right\vert ^{\alpha }  \notag \\
&&+\frac{C}{t}e^{-\sqrt{\lambda _{1}\left( M\right) }R}+Ct^{\beta -\frac{1}{%
n-1}},  \notag
\end{eqnarray}%
where 
\begin{eqnarray*}
\Lambda _{1} &=&\beta \left( 1-\beta \right) -\frac{\alpha \left( n-2\right) 
}{\alpha +2}A^{-2} \\
\Lambda _{2} &=&\frac{2\left( n-2\right) }{\alpha +2}A^{\alpha }.
\end{eqnarray*}%
We apply Proposition \ref{P} for $K=M\setminus D$ and  
$\phi =u^{\frac{\alpha +\beta }{2}}\chi \psi$. As $\partial K=\ell(1)$,  and $\nu= \frac{\nabla u}{\vert \nabla u\vert}$ is the
outward unit normal to $K$,
we obtain that 
\begin{equation*}
\lambda _{1}\left( M\right) \int_{M\setminus D}u^{\alpha +\beta }\chi
^{2}\psi ^{2}\leq \int_{M\setminus D}\left\vert \nabla \left( u^{\frac{%
\alpha +\beta }{2}}\chi \psi \right) \right\vert ^{2}-\int_{\ell\left( 1\right)
}h_{\nu }.
\end{equation*}%
Note that 
\begin{eqnarray*}
\int_{M\setminus D}\left\vert \nabla \left( u^{\frac{\alpha +\beta }{2}}\chi
\psi \right) \right\vert ^{2} &=&\frac{\left( \alpha +\beta \right) ^{2}}{4}%
\int_{M\setminus D}\left\vert \nabla u\right\vert ^{2}u^{\alpha +\beta
-2}\chi ^{2}\psi ^{2} \\
&&+\int_{M\setminus D}u^{\alpha +\beta }\left\vert \nabla \left( \chi \psi
\right) \right\vert ^{2} \\
&&+\frac{1}{2}\int_{M\setminus D}\left\langle \nabla u^{\alpha +\beta
},\nabla \left( \chi \psi \right) ^{2}\right\rangle .
\end{eqnarray*}%
By (\ref{u}), (\ref{alpha}), and (\ref{beta}) we conclude that 
\begin{eqnarray*}
\int_{M\setminus D}u^{\alpha +\beta }\left\vert \nabla \psi \right\vert
^{2}\chi ^{2} &\leq &\int_{L\left( \frac{1}{2}t,1\right) \setminus
B_{p}\left( R\right) }u^{\alpha +\beta }\chi ^{2} \\
&\leq &\frac{2}{t}\int_{L\left( \frac{1}{2}t,1\right) \setminus B_{p}\left(
R\right) }u^{2}\chi ^{2} \\
&\leq &\frac{C}{t}e^{-2\sqrt{\lambda _{1}\left( M\right) }R}.
\end{eqnarray*}%
Similarly, by additionally using (\ref{gradest}) we get that 
\begin{equation*}
\frac{1}{2}\int_{M\setminus D}\left\langle \nabla u^{\alpha +\beta },\nabla
\psi ^{2}\right\rangle \chi ^{2}\leq \frac{C}{t}e^{-2\sqrt{\lambda
_{1}\left( M\right) }R}.
\end{equation*}%
By (\ref{x5}), (\ref{alpha}), and (\ref{beta}) we have 
\begin{eqnarray*}
\int_{M}u^{\alpha +\beta }\left\vert \nabla \chi \right\vert ^{2}\psi ^{2}
&\leq &C\int_{L\left( \frac{1}{2}t,t\right) }\left\vert \nabla u\right\vert
^{2}u^{\beta -\frac{n}{n-1}} \\
&\leq &\frac{C}{\beta -\frac{1}{n-1}}t^{\beta -\frac{1}{n-1}}.
\end{eqnarray*}%
Similarly, 
\begin{equation*}
\int_{M}\left\langle \nabla u^{\alpha +\beta },\nabla \chi ^{2}\right\rangle
\psi ^{2}\leq \frac{C}{\beta -\frac{1}{n-1}}t^{\beta -\frac{1}{n-1}}.
\end{equation*}%
Combining all these estimates we arrive at
\begin{eqnarray}
&&\lambda _{1}\left( M\right) \int_{M\setminus D}u^{\alpha +\beta }\chi
^{2}\psi ^{2}  \label{m9} \\
&\leq &\frac{\left( \alpha +\beta \right) ^{2}}{4}\int_{M\setminus
D}\left\vert \nabla u\right\vert ^{2}u^{\alpha +\beta -2}\chi ^{2}\psi ^{2} 
\notag \\
&&-\int_{\ell\left( 1\right) }h_{\nu }+\frac{C}{t}e^{-2\sqrt{\lambda _{1}\left(
M\right) }R}+Ct^{\beta -\frac{1}{n-1}}.  \notag
\end{eqnarray}
By Young's inequality, 
\begin{equation*}
\left\vert \nabla u\right\vert ^{2}\leq \frac{2}{\alpha +2}B^{-\alpha
}\left\vert \nabla u\right\vert ^{\alpha +2}u^{-\alpha }+\frac{\alpha }{%
\alpha +2}B^{2}u^{2}
\end{equation*}%
for a constant $B>0$ to be specified later. This yields 
\begin{eqnarray*}
\int_{M\setminus D}\left\vert \nabla u\right\vert ^{2}u^{\alpha +\beta
-2}\chi ^{2}\psi ^{2} &\leq &\frac{2}{\alpha +2}B^{-\alpha }\int_{M\setminus
D}\left\vert \nabla u\right\vert ^{\alpha +2}u^{\beta -2}\chi ^{2}\psi ^{2}
\\
&&+\frac{\alpha }{\alpha +2}B^{2}\int_{M\setminus D}u^{\alpha +\beta }\chi
^{2}\psi ^{2}.
\end{eqnarray*}%
Plugging this into (\ref{m9}), one concludes that

\begin{eqnarray*}
&&\left( \lambda _{1}\left( M\right) -\frac{\alpha }{\alpha +2}\frac{\left(
\alpha +\beta \right) ^{2}}{4}B^{2}\right) \int_{M\setminus D}u^{\alpha
+\beta }\chi ^{2}\psi ^{2} \\
&\leq &\frac{\left( \alpha +\beta \right) ^{2}}{2\left( \alpha +2\right) }%
B^{-\alpha }\int_{M\setminus D}\left\vert \nabla u\right\vert ^{\alpha
+2}u^{\beta -2}\chi ^{2}\psi ^{2} \\
&&-\int_{\ell\left( 1\right) }h_{\nu }+\frac{C}{t}e^{-2\sqrt{\lambda _{1}\left(
M\right) }R}+Ct^{\beta -\frac{1}{n-1}}.
\end{eqnarray*}

By the assumption, 
\begin{equation*}
\lambda _{1}\left( M\right) \geq \frac{\left( n-1\right) ^{2}\delta ^{2}}{4}
\end{equation*}%
with
\begin{equation}
\delta =\frac{2\left( n-2\right) }{\left( n-1\right) ^{2}}\sqrt{2n-3}.
\label{delta}
\end{equation}%
Therefore,

\begin{eqnarray*}
&&\left( \frac{\left( n-1\right) ^{2}\delta ^{2}}{4}-\frac{\alpha }{\alpha +2%
}\frac{\left( \alpha +\beta \right) ^{2}}{4}B^{2}\right) \int_{M\setminus
D}u^{\alpha +\beta }\chi ^{2}\psi ^{2} \\
&\leq &\frac{\left( \alpha +\beta \right) ^{2}}{2\left( \alpha +2\right) }%
B^{-\alpha }\int_{M\setminus D}\left\vert \nabla u\right\vert ^{\alpha
+2}u^{\beta -2}\chi ^{2}\psi ^{2} \\
&&-\int_{\ell\left( 1\right) }h_{\nu }+\frac{C}{t}e^{-2\sqrt{\lambda _{1}\left(
M\right) }R}+Ct^{\beta -\frac{1}{n-1}}.
\end{eqnarray*}%
We optimize this inequality by choosing 
\begin{equation*}
B=\frac{\left( n-1\right) \delta }{\alpha +\beta }
\end{equation*}%
and obtain that 
\begin{eqnarray*}
\int_{M\setminus D}u^{\alpha +\beta }\chi ^{2}\psi ^{2} &\leq &\left( \frac{%
\alpha +\beta }{\left( n-1\right) \delta }\right) ^{\alpha
+2}\int_{M\setminus D}\left\vert \nabla u\right\vert ^{\alpha +2}u^{\beta
-2}\chi ^{2}\psi ^{2} \\
&&-\frac{2\left( \alpha +2\right) }{\left( n-1\right) ^{2}\delta ^{2}}%
\int_{\ell\left( 1\right) }h_{\nu } \\
&&+\frac{C}{t}e^{-2\sqrt{\lambda _{1}\left( M\right) }R}+Ct^{\beta -\frac{1}{%
n-1}}.
\end{eqnarray*}

Plugging this into (\ref{m8}), we conclude that

\begin{eqnarray}
\Lambda \int_{M\setminus D}\left\vert \nabla u\right\vert ^{\alpha
+2}u^{\beta -2}\chi ^{2}\psi ^{2} &\leq &-\frac{2\left( \alpha +2\right) }{%
\left( n-1\right) ^{2}\delta ^{2}}\Lambda _{2}\int_{l\left( 1\right) }h_{\nu
}  \label{m10} \\
&&+\int_{\ell\left( 1\right) }\left( \left\vert \nabla u\right\vert ^{\alpha
}\right) _{\nu }-\int_{\ell\left( 1\right) }\left( u^{\beta }\right) _{\nu
}\left\vert \nabla u\right\vert ^{\alpha }  \notag \\
&&+\frac{C}{t}e^{-\sqrt{\lambda _{1}\left( M\right) }R}+Ct^{\beta -\frac{1}{%
n-1}},  \notag
\end{eqnarray}%
where%
\begin{equation*}
\Lambda _{2}=\frac{2\left( n-2\right) }{\alpha +2}A^{\alpha }
\end{equation*}%
and 
\begin{eqnarray*}
\Lambda  &=&\Lambda _{1}-\left( \frac{\alpha +\beta }{\left( n-1\right)
\delta }\right) ^{\alpha +2}\Lambda _{2} \\
&=&\beta \left( 1-\beta \right) -\frac{\alpha \left( n-2\right) }{\alpha +2}%
A^{-2} \\
&&-\left( \frac{\alpha +\beta }{\left( n-1\right) \delta }\right) ^{\alpha
+2}\frac{2\left( n-2\right) }{\alpha +2}A^{\alpha }.
\end{eqnarray*}%
We optimize $\Lambda $ by choosing 
\begin{equation*}
A=\frac{\left( n-1\right) \delta }{\alpha +\beta }
\end{equation*}%
and obtain that 
\begin{equation}
\Lambda =\beta \left( 1-\beta \right) -\frac{\left( n-2\right) \left( \alpha
+\beta \right) ^{2}}{\left( n-1\right) ^{2}\delta ^{2}}.  \label{Lambda}
\end{equation}
Hence, (\ref{m10}) becomes

\begin{eqnarray}
\Lambda \int_{M\setminus D}\left\vert \nabla u\right\vert ^{\alpha
+2}u^{\beta -2}\chi ^{2}\psi ^{2} &\leq &-\frac{4\left( n-2\right) }{\left(
\alpha +\beta \right) ^{\alpha }\left( \left( n-1\right) \delta \right)
^{2-\alpha }}\int_{\ell\left( 1\right) }h_{\nu }  \label{m101} \\
&&+\int_{\ell\left( 1\right) }\left( \left\vert \nabla u\right\vert ^{\alpha
}\right) _{\nu }-\int_{\ell\left( 1\right) }\left( u^{\beta }\right) _{\nu
}\left\vert \nabla u\right\vert ^{\alpha }  \notag \\
&&+\frac{C}{t}e^{-\sqrt{\lambda _{1}\left( M\right) }R}+Ct^{\beta -\frac{1}{%
n-1}}.  \notag
\end{eqnarray}

Recall that 
\begin{eqnarray*}
\alpha  &=&\frac{n-2}{n-1} \\
\delta ^{2} &=&\frac{4\left( n-2\right) ^{2}}{\left( n-1\right) ^{4}}\left(
2n-3\right) ,
\end{eqnarray*}%
as specified in (\ref{alpha}) and (\ref{delta}). We let 
\begin{equation}
\beta =\frac{n-2}{3n-5}.  \label{beta1}
\end{equation}%
Note that for any $n\geq 5$ we have $\beta >\frac{1}{n-1}$, as required by (\ref%
{beta}). 

Furthermore, it follows that $\Lambda =0$ by direct calculation. Consequently, (\ref{m101}) reduces to

\begin{eqnarray}
0 &\leq &-\frac{4\left( n-2\right) }{\left( \alpha +\beta \right) ^{\alpha
}\left( \left( n-1\right) \delta \right) ^{2-\alpha }}\int_{\ell\left( 1\right)
}h_{\nu }  \label{m11} \\
&&+\int_{\ell\left( 1\right) }\left( \left\vert \nabla u\right\vert ^{\alpha
}\right) _{\nu }-\int_{\ell\left( 1\right) }\left( u^{\beta }\right) _{\nu
}\left\vert \nabla u\right\vert ^{\alpha }  \notag \\
&&+\frac{C}{t}e^{-\sqrt{\lambda _{1}\left( M\right) }R}+Ct^{\beta -\frac{1}{%
n-1}}.  \notag
\end{eqnarray}

However, by (\ref{x2}),
\begin{eqnarray*}
\lambda _{1}\left( M\right) \mathrm{Vol}\left( D\cap B_{p}\left( R\right)
\right) &\leq &\lambda _{1}\left( M\right) \int_{D}\psi ^{2} \\
&=&-\int_{D}\left( \Delta h+\left\vert \nabla h\right\vert ^{2}\right) \psi
^{2} \\
&=&\int_{\ell\left( 1\right) }h_{\nu }+\int_{D}\left\langle \nabla h,\nabla
\psi ^{2}\right\rangle \\
&&-\int_{D}\left\vert \nabla h\right\vert ^{2}\psi ^{2} \\
&\leq &\int_{\ell\left( 1\right) }h_{\nu }+\int_{D}\left\vert \nabla \psi
\right\vert ^{2},
\end{eqnarray*}%
where $\nu =\frac{\nabla u}{\left\vert \nabla u\right\vert }$ is the inward
unit normal to $\partial D=\ell\left( 1\right).$ According to (\ref{VolFinite}) it
follows that 
\begin{equation*}
\lambda _{1}\left( M\right) \mathrm{Vol}\left( D\cap B_{p}\left( R\right)
\right) \leq \int_{\ell\left( 1\right) }h_{\nu }+Ce^{-2\sqrt{\lambda _{1}\left(
M\right) }R}.
\end{equation*}%
Making $R\rightarrow \infty $ yields 
\begin{equation*}
\lambda _{1}\left( M\right) \mathrm{Vol}\left( D\right) \leq \int_{\ell\left(
1\right) }h_{\nu }.
\end{equation*}%
As $\Omega \subset D,$ we conclude from (\ref{m11}) that 
\begin{eqnarray}
&&\frac{4\left( n-2\right) \lambda _{1}\left( M\right) }{\left( \alpha
+\beta \right) ^{\alpha }\left( \left( n-1\right) \delta \right) ^{2-\alpha }%
}\mathrm{Vol}\left( \Omega \right)  \label{m12} \\
&\leq &\int_{\ell\left( 1\right) }\left( \left\vert \nabla u\right\vert
^{\alpha }\right) _{\nu }-\int_{\ell\left( 1\right) }\left( u^{\beta }\right)
_{\nu }\left\vert \nabla u\right\vert ^{\alpha }  \notag \\
&&+\frac{C}{t}e^{-\sqrt{\lambda _{1}\left( M\right) }R}+Ct^{\beta -\frac{1}{%
n-1}}.  \notag
\end{eqnarray}

Letting $R\rightarrow \infty $ first and then $t\rightarrow 0$ in (\ref{m12})
we arrive at 
\begin{equation*}
\frac{4\left( n-2\right) \lambda _{1}\left( M\right) }{\left( \alpha +\beta
\right) ^{\alpha }\left( \left( n-1\right) \delta \right) ^{2-\alpha }}%
\mathrm{Vol}\left( \Omega \right) \leq \int_{\ell\left( 1\right) }\left(
\left\vert \nabla u\right\vert ^{\alpha }\right) _{\nu }-\int_{\ell\left(
1\right) }\left( u^{\beta }\right) _{\nu }\left\vert \nabla u\right\vert
^{\alpha }.
\end{equation*}%
Note that the mean curvature of 
\begin{equation*}
\Sigma =\ell\left( 1\right) =\partial \left( M\setminus D\right) 
\end{equation*}
satisfies%
\begin{equation*}
H_{\Sigma }=-\frac{\left\langle \nabla \left\vert \nabla u\right\vert
,\nabla u\right\rangle }{\left\vert \nabla u\right\vert ^{2}}.
\end{equation*}%
Hence,

\begin{equation*}
\left( \left\vert \nabla u\right\vert ^{\alpha }\right) _{\nu }=-\alpha
H_{\Sigma }\left\vert \nabla u\right\vert ^{\alpha }.
\end{equation*}%
This proves that 
\begin{equation*}
\frac{4\left( n-2\right) \lambda _{1}\left( M\right) }{\left( \alpha +\beta
\right) ^{\alpha }\left( \left( n-1\right) \delta \right) ^{2-\alpha }}%
\mathrm{Vol}\left( \Omega \right) \leq -\alpha \int_{\Sigma }H_{\Sigma
}\left\vert \nabla u\right\vert ^{\alpha }-\beta \int_{\Sigma }\left\vert
\nabla u\right\vert ^{\alpha +1}.
\end{equation*}%
From Young's inequality that
\begin{equation*}
\alpha \left\vert H_{\Sigma }\right\vert \left\vert \nabla u\right\vert
^{\alpha }\leq \beta \left\vert \nabla u\right\vert ^{\alpha +1}+\frac{%
\alpha }{\alpha +1}\left( \frac{\alpha ^{2}}{\left( \alpha +1\right) \beta }%
\right) ^{\alpha }\left\vert H_{\Sigma }\right\vert ^{\alpha +1},
\end{equation*}%
we conclude 

\begin{equation}
\Gamma \,\mathrm{Vol}\left( \Omega \right) \leq \int_{\Sigma }\left\vert
H_{\Sigma }\right\vert ^{\alpha +1}\leq \int_{\partial \Omega }\left\vert
H\right\vert ^{\alpha +1},  \label{m13}
\end{equation}%
where%
\begin{equation*}
\Gamma =\frac{4\left( n-2\right) }{\left( \alpha +\beta \right) ^{\alpha
}\left( \left( n-1\right) \delta \right) ^{2-\alpha }}\frac{\alpha +1}{%
\alpha }\left( \frac{\left( \alpha +1\right) \beta }{\alpha ^{2}}\right)
^{\alpha }\lambda _{1}\left( M\right) 
\end{equation*}%
and
\begin{eqnarray*}
\alpha  &=&\frac{n-2}{n-1},\ \ \ \ \beta =\frac{n-2}{3n-5}, \\
\delta  &=&\frac{2\left( n-2\right) }{\left( n-1\right) ^{2}}\sqrt{2n-3}.
\end{eqnarray*}%
For $n\geq 5$ we have that 
\begin{eqnarray*}
\delta  &<&1, \\
\ \ \alpha +\beta  &<&\frac{4}{3}, \\
\frac{\left( \alpha +1\right) \beta }{\alpha ^{2}} &>&\frac{1}{2}, \\
\ \ \left( n-1\right) ^{2-\alpha } &<&\sqrt{2}\left( n-1\right) .
\end{eqnarray*}%
Therefore,
\begin{equation*}
\Gamma >\frac{2}{3}\sqrt{n}\ \lambda _{1}\left( M\right).
\end{equation*}%
In conclusion, by (\ref{m13}) we have 
\begin{equation*}
\frac{2}{3}\sqrt{n}\ \lambda _{1}\left( M\right) \mathrm{Vol}\left( \Omega
\right) \leq \int_{\partial \Omega }\left\vert H\right\vert ^{\frac{2n-3}{n-1%
}}.
\end{equation*}
\end{proof}

\section{Splitting of Ricci solitons}

In this section, we address the issue of nonexistence of compact minimal hypersurfaces 
in Ricci solitons. We begin with the case of steady solitons.
Let $\left( M,g,f\right) $ be a gradient steady Ricci soliton. Then the potential 
$f$ satisfies the soliton equation 
\begin{equation*}
\mathrm{Ric}+\mathrm{Hess}\left( f\right) =0.
\end{equation*}%
It is known \cite{H} that $f$ may be normalized so that 
\begin{equation*}
S+\left\vert \nabla f\right\vert ^{2}=1,
\end{equation*}%
where $S$ is the scalar curvature. It is also known \cite{Chen} that $S>0$ unless $\left(
M,g\right) $ is Ricci flat.

\begin{theorem}
\label{Steady}Let $\left( M^{n},g,f\right) $ be a steady Ricci soliton.
Assume that there exists a smooth compact embedded minimal hypersurface $\Sigma $ in $M.$
Then $\left( M,g\right) $ splits isometrically as a direct product $\mathbb{R}\times \Sigma.$
\end{theorem}

\begin{proof}
By the splitting theorem in \cite{MW1} we may assume that $M$ and its
double covers all have one end. Hence, according to Proposition 5.2
in \cite{CP}, the integral homology 
\begin{equation*}
H_{n-1}\left( M,\mathbb{Z}\right) =\left\{ 0\right\}.
\end{equation*}%
In particular, $\Sigma $ bounds a compact domain $D$ in $M.$

In \cite{MW1} it was proved that $\Delta _{f}$ has positive spectrum. Consequently,
$M$ is $f$-nonparabolic. This implies that there exists $w>0$ on $M\setminus D$ such
that 
\begin{eqnarray*}
\Delta _{f}w &=&0\text{ \ on }M\setminus D \\
w &=&1\text{ on }\Sigma \\
\inf_{M\setminus D}w &=&0.
\end{eqnarray*}%
Moreover, 
\begin{equation}
\int_{M\setminus D}\left\vert \nabla w\right\vert ^{2}e^{-f}<\infty.
\label{z2}
\end{equation}%
The Bochner formula implies that 
\begin{equation}
\frac{1}{2}\Delta _{f}\left\vert \nabla w\right\vert ^{2}\geq \left\vert
\nabla \left\vert \nabla w\right\vert \right\vert ^{2}.  \label{z3}
\end{equation}
We now prove, similar to Proposition \ref{P}, that 
\begin{equation}
0\leq \int_{M\setminus D}\left\vert \nabla \phi \right\vert
^{2}e^{-f}-\int_{\Sigma }f_{\nu }\phi ^{2}e^{-f}  \label{z1}
\end{equation}%
for any smooth function $\phi$ on $M\setminus D$ that vanishes at infinity,
where $\nu =\frac{\nabla w}{\left\vert \nabla w\right\vert }$ is the outward unit normal
to $\partial \left(M\setminus D\right)=\Sigma =\left\{ w=1\right\}.$

Note that
\begin{eqnarray*}
\Delta _{f}\left( f\right) &=&\Delta f-\left\vert \nabla f\right\vert ^{2} \\
&=&-S-\left\vert \nabla f\right\vert ^{2} \\
&=&-1.
\end{eqnarray*}%
Therefore,
\begin{eqnarray*}
\int_{M\setminus D}\phi ^{2}e^{-f} &=&-\int_{M\setminus D}\left( \Delta
_{f}\left( f\right) \right) \phi ^{2}e^{-f} \\
&=&\int_{M\setminus D}\left\langle \nabla \phi ^{2},\nabla f\right\rangle
e^{-f}-\int_{\Sigma }f_{\nu }\phi ^{2}e^{-f} \\
&\leq &\int_{M\setminus D}\phi ^{2}e^{-f}+\int_{M\setminus D}\left\vert
\nabla \phi \right\vert ^{2}e^{-f} \\
&&-\int_{\Sigma }f_{\nu }\phi ^{2}e^{-f}.
\end{eqnarray*}%
This proves (\ref{z1}).

Let $\psi $ be a smooth function on $M\setminus D$ with $\psi =1$ on $\Sigma$ 
and $\psi=0$ outside a sufficiently large ball $B_p(R).$
Setting $\phi=\left\vert \nabla w\right\vert \psi $ in (\ref{z1}) we get that 
\begin{eqnarray*}
0 &\leq &\int_{M\setminus D}\left\vert \nabla \left( \left\vert \nabla
w\right\vert \psi \right) \right\vert ^{2}e^{-f}-\int_{\Sigma }f_{\nu
}\left\vert \nabla w\right\vert ^{2}e^{-f} \\
&=&\int_{M\setminus D}\left\vert \nabla \left\vert \nabla w\right\vert
\right\vert ^{2}\psi ^{2}e^{-f}+\frac{1}{2}\int_{M\setminus D}\left\langle
\nabla \left\vert \nabla w\right\vert ^{2},\nabla \psi ^{2}\right\rangle e^{-f}\\
&&+\int_{M\setminus D}\left\vert \nabla \psi \right\vert ^{2}\left\vert
\nabla w\right\vert ^{2}e^{-f}-\int_{\Sigma }f_{\nu }\left\vert \nabla
w\right\vert ^{2}e^{-f}.
\end{eqnarray*}%
By (\ref{z3}), this yields 
\begin{eqnarray}
0 &\leq &\frac{1}{2}\int_{M\setminus D}\left( \Delta _{f}\left\vert \nabla
w\right\vert ^{2}\right) \psi ^{2}+\frac{1}{2}\int_{M\setminus
D}\left\langle \nabla \left\vert \nabla w\right\vert ^{2},\nabla \psi
^{2}\right\rangle  e^{-f}\label{z5} \\
&&+\int_{M\setminus D}\left\vert \nabla \psi \right\vert ^{2}\left\vert
\nabla w\right\vert ^{2}e^{-f}-\int_{\Sigma }f_{\nu }\left\vert \nabla
w\right\vert ^{2}e^{-f}  \notag \\
&=&\frac{1}{2}\int_{\Sigma }\left( \left\vert \nabla w\right\vert
^{2}\right) _{\nu }e^{-f}+\int_{M\setminus D}\left\vert \nabla \psi
\right\vert ^{2}\left\vert \nabla w\right\vert ^{2}e^{-f}  \notag \\
&&-\int_{\Sigma }f_{\nu }\left\vert \nabla w\right\vert ^{2}e^{-f}.  \notag
\end{eqnarray}%
However, as $\nu =\frac{\nabla w}{\left\vert \nabla w\right\vert }$ and $%
H_{\Sigma }=0,$ we see that 
\begin{eqnarray*}
\frac{1}{2}\left( \left\vert \nabla w\right\vert ^{2}\right) _{\nu }
&=&\left\langle \nabla \left\vert \nabla w\right\vert ,\nabla w\right\rangle
\\
&=&\left( \Delta w\right) \left\vert \nabla w\right\vert \\
&=&\left\langle \nabla f,\nabla w\right\rangle \left\vert \nabla w\right\vert
\\
&=&f_{\nu }\left\vert \nabla w\right\vert ^{2}.
\end{eqnarray*}%
Hence, (\ref{z5}) becomes an equality by letting $R\to \infty,$ which in turn forces (\ref{z3}) to be an
equality. This implies the splitting of the manifold as a direct product $\mathbb{R}\times \Sigma.$
We refer to \cite{MW1} for the details.
\end{proof}

An analogous result for expanding Ricci solitons holds true as well.
Recall that an expanding Ricci soliton $\left( M,g,f\right) $ satisfies the
equation 
\begin{equation*}
\mathrm{Ric}+\mathrm{Hess}\left( f\right) =-\frac{1}{2}g.
\end{equation*}%
We may normalize $f$ (see \cite{H}) such that 
\begin{equation*}
S+\left\vert \nabla f\right\vert ^{2}=-f.
\end{equation*}%
Moreover, the scalar curvature $S\geq -\frac{n}{2}$ on $M$ by \cite{PRS}.

\begin{theorem}
\label{Expander}Let $\left( M^{n},g,f\right) $ be an expanding gradient Ricci soliton
with $S\geq -\frac{n-1}{2}$ on $M.$ Assume that there exists an embedded compact minimal
hypersurface $\Sigma $ in $M.$ Then $M$ splits isometrically as a direct
product $\mathbb{R}\times \Sigma.$
\end{theorem}

\begin{proof}
Recall that by \cite{MW2} such an expanding Ricci soliton must have one end
or it splits as a direct product. Hence, by \cite{CP} we may assume as
before that the integral homology
\begin{equation*}
H_{n-1}\left( M,\mathbb{Z}\right) =\left\{ 0\right\}
\end{equation*}%
and that $\Sigma $ bounds a compact domain $D$ in $M.$

As $\Delta _{f}$ has positive spectrum \cite{MW2}, $M$ is $f$-nonparabolic. 
In particular, there exists function $w>0$ on $M\setminus D$ such that 
\begin{eqnarray*}
\Delta _{f}w &=&0\text{ \ on }M\setminus D \\
w &=&1\text{ on }\Sigma \\
\inf_{M\setminus D}w &=&0.
\end{eqnarray*}%
Moreover, 
\begin{equation*}
\int_{M\setminus D}\left\vert \nabla w\right\vert
^{2}e^{-f}<\infty.
\end{equation*}

We now prove, similar to Proposition \ref{P}, that 
\begin{equation}
\frac{1}{2}\int_{M\setminus D}\phi ^{2}e^{-f}\leq \int_{M\setminus
D}\left\vert \nabla \phi \right\vert ^{2}e^{-f}-\int_{\Sigma }f_{\nu }\phi
^{2}e^{-f} \label{z6}
\end{equation}%
for any smooth function $\phi$ on $M\setminus D$ that vanishes near infinity,
where $\nu =\frac{\nabla w}{\left\vert \nabla w\right\vert }$ is the unit normal
to $\Sigma =\left\{ w=1\right\}.$ 
Direct calculation gives
\begin{eqnarray*}
\Delta _{f}\left( f\right) &=&\Delta f-\left\vert \nabla f\right\vert ^{2} \\
&=&-\frac{n}{2}-S-\left\vert \nabla f\right\vert ^{2} \\
&\leq &-\frac{1}{2}-\left\vert \nabla f\right\vert ^{2}.
\end{eqnarray*}%
So we have
\begin{eqnarray*}
\frac{1}{2}\int_{M\setminus D}\phi ^{2}e^{-f} &\leq&-\int_{M\setminus D}\left(
\Delta _{f}\left( f\right) \right) \phi ^{2}e^{-f}-\int_{M\setminus
D}\left\vert \nabla f\right\vert ^{2}\phi ^{2}e^{-f} \\
&=&\int_{M\setminus D}\left\langle \nabla \phi ^{2},\nabla f\right\rangle
e^{-f}-\int_{\Sigma }f_{\nu }\phi ^{2}e^{-f}-\int_{M\setminus D}\left\vert
\nabla f\right\vert ^{2}\phi ^{2}e^{-f} \\
&\leq &\int_{M\setminus D}\left\vert \nabla \phi \right\vert
^{2}e^{-f}-\int_{\Sigma }f_{\nu }\phi^2 e^{-f}.
\end{eqnarray*}%
This proves (\ref{z6}). We apply (\ref{z6}) to $\phi =\left\vert \nabla
w\right\vert \psi,$ where $\psi $ is a cut-off function as in 
Theorem \ref{Steady} that $\psi=1$ on $\Sigma$ and that $\psi=0$ outside the geodesic ball
$B_p(R)$ when $R$ is large. It follows that 
\begin{eqnarray}
\frac{1}{2}\int_{M\setminus D}\left\vert \nabla w\right\vert ^{2}\psi
^{2}e^{-f} &\leq &\int_{M\setminus D}\left\vert \nabla \left\vert \nabla
w\right\vert \right\vert ^{2}\psi ^{2}e^{-f}  \label{z7} \\
&&+\int_{M\setminus D}\left\vert \nabla w\right\vert ^{2}\left\vert \nabla
\psi \right\vert ^{2}e^{-f}  \notag \\
&&+\frac{1}{2}\int_{M\setminus D}\left\langle \nabla \left\vert \nabla
w\right\vert ^{2},\nabla \psi ^{2}\right\rangle e^{-f}  \notag \\
&&-\int_{\Sigma }f_{\nu } \vert \nabla w\vert ^2  e^{-f}.  \notag
\end{eqnarray}%
Recall the Bochner formula 
\begin{equation*}
\frac{1}{2}\Delta _{f}\left\vert \nabla w\right\vert ^{2}\geq \left\vert
\nabla \left\vert \nabla w\right\vert \right\vert ^{2}-\frac{1}{2}\left\vert
\nabla w\right\vert ^{2}\text{ \ on }M\setminus D.
\end{equation*}%
Plugging into (\ref{z7}) yields that

\begin{eqnarray*}
\frac{1}{2}\int_{M\setminus D}\left\vert \nabla w\right\vert ^{2}\psi
^{2}e^{-f} &\leq &\frac{1}{2}\int_{M\setminus D}\left( \Delta _{f}\left\vert
\nabla w\right\vert ^{2}\right) \psi ^{2}e^{-f}+\frac{1}{2}\int_{M\setminus
D}\left\vert \nabla w\right\vert ^{2}\psi ^{2}e^{-f} \\
&&+\int_{M\setminus D}\left\vert \nabla w\right\vert ^{2}\left\vert \nabla
\psi \right\vert ^{2}e^{-f}+\frac{1}{2}\int_{M\setminus D}\left\langle
\nabla \left\vert \nabla w\right\vert ^{2},\nabla \psi ^{2}\right\rangle
e^{-f} \\
&&-\int_{\Sigma }f_{\nu }\left\vert \nabla w\right\vert ^{2}e^{-f} \\
&=&\frac{1}{2}\int_{\Sigma }\left( \left\vert \nabla w\right\vert
^{2}\right) _{\nu }e^{-f}+\frac{1}{2}\int_{M\setminus D}\left\vert
\nabla w\right\vert ^{2}\psi ^{2}e^{-f} \\
&&+\int_{M\setminus D}\left\vert \nabla w\right\vert ^{2}\left\vert \nabla
\psi \right\vert ^{2}e^{-f}-\int_{\Sigma }f_{\nu }\left\vert \nabla
w\right\vert ^{2}e^{-f}.
\end{eqnarray*}
In conclusion, 
\begin{eqnarray*}
0 &\leq &\frac{1}{2}\int_{\Sigma }\left( \left\vert \nabla w\right\vert
^{2}\right) _{\nu }e^{-f}-\int_{\Sigma }f_{\nu }\left\vert \nabla
w\right\vert ^{2}e^{-f} \\
&&+\int_{M\setminus D}\left\vert \nabla w\right\vert ^{2}\left\vert \nabla
\psi \right\vert ^{2}e^{-f}.
\end{eqnarray*}%
Making $R\rightarrow \infty,$ we get
\begin{equation*}
0\leq \frac{1}{2}\int_{\Sigma }\left( \left\vert \nabla w\right\vert
^{2}\right) _{\nu }e^{-f}-\int_{\Sigma }f_{\nu }\left\vert \nabla
w\right\vert ^{2}e^{-f}.
\end{equation*}%
Since $\nu =\frac{\nabla w}{\left\vert \nabla w\right\vert },$ it follows
that 
\begin{equation*}
0\leq \int_{\Sigma }\left\langle \nabla \left\vert \nabla w\right\vert
,\nabla w\right\rangle e^{-f}-\int_{\Sigma }\left\langle \nabla f,\nabla
w\right\rangle \left\vert \nabla w\right\vert e^{-f}.
\end{equation*}%
However, as $w$ is $f$-harmonic, 
\begin{equation*}
\left\langle \nabla \left\vert \nabla w\right\vert ,\nabla w\right\rangle
-\left\langle \nabla f,\nabla w\right\rangle \left\vert \nabla w\right\vert
=-H_{\Sigma }\left\vert \nabla w\right\vert ^{2}.
\end{equation*}%
Since $\Sigma $ is minimal, this again shows the above inequality must be equality,
which in turn forces the Bochner formula itself is also an equality. This suffices
to conclude that $M=\mathbb{R}\times \Sigma.$ One may refer to \cite{MW2} for details.
\end{proof}

\textbf{Acknowledgment:} We wish to thank Pengfei Guan for his interest and  comments. The first author was
partially supported by the NSF grant DMS-1811845 and by a Simons Foundation grant.


\begin{thebibliography}{99}
\bibitem{AMO} V. Agostiniani, M. Fogagnolo and L. Mazzieri, Sharp geometric
inequalities for closed hypersurfaces in manifolds with nonnegative Ricci
curvature, Invent. Math. 222 (2020), 1033-1101.

\bibitem{AMO2} V. Agostiniani, M. Fogagnolo and L. Mazzieri,
Minkowski inequalities via nonlinear potential theory,
Arch. Ration. Mech. Anal. 244 (2022), no. 1, 51-85.

\bibitem{BHW} S. Brendle, P. Hung and M. Wang, 
A Minkowski inequality for hypersurfaces in the Anti-de Sitter-Schwarzschild manifold, 
Comm. Pure Appl. Math. 69 (2016), no. 1, 124-144.
 
\bibitem{Brooks} R. Brooks, The fundamental group and the spectrum of the
Laplacian, Comment. Math. Helv. 56 (1981), no. 4, 581-598.

\bibitem{CG} J. Cheeger and D. Gromoll, The splitting theorem for manifolds
of nonnegative Ricci curvature, J. Differential Geometry 6 (1971), 119-128.

\bibitem{CP} G. Carron and E. Pedon, On the differential form spectrum of
hyperbolic manifolds, Ann. Sc. Norm. Sup. Pisa, 3 (2004), no. 4, 705-747.

\bibitem{CW} A. Chang and Y. Wang,
Inequalities for quermassintegrals on k-convex domains,
Adv. Math. 248 (2013), 335-377.

\bibitem{Chen} B.L. Chen, Strong uniqueness of the Ricci flow, J.
Differential Geom. 82 (2009), no. 2, 362-382.

\bibitem{C} S.Y. Cheng, Eigenvalue comparison theorems and its geometric
applications, Math. Z. 143 (1975), 289-297

\bibitem{CLN} B. Chow, P. Lu and L. Ni, Hamilton's Ricci flow, Graduate
studies in mathematics, 2006.

\bibitem{Co} T. Colding, New monotonicity formulas for Ricci curvature and
applications I, Acta Math. 209 (2012) 229-263.

\bibitem{CM} T. Colding and W. Minicozzi, Ricci curvature and monotonicity
for harmonic functions, Calc. Var. Partial Differ. Equ. 49 (2014), 1045-1059.


\bibitem {dG} L. L. de Lima and F. Gir\~{a}o, 
An Alexandrov-Fenchel-type Inequality in Hyperbolic Space with an application to a Penrose Inequality,
Ann. Henri Poincar\'e 17(4), 979-1002, 2016.

\bibitem{FS} D. Fischer-Colbrie and R. Schoen, The structure of complete
stable minimal surfaces in 3-manifolds of nonnegative scalar curvature,
Comm. Pure Appl. Math. 33 (1980), no. 2, 199-211.

\bibitem{Gaffney} M. Gaffney, A special Stokes's theorem for complete Riemannian manifolds,
Ann. of Math. (2) 60 (1954), 140-145.

\bibitem{GWW} Y. Ge, G. Wang and J. Wu, 
Hyperbolic Alexandrov-Fenchel quermassintegral inequalities II,
J. Differential Geom. 98 (2014), 237-260.

\bibitem{GL1} P. Guan and J. Li,
The quermassintegral inequalities for k-convex starshaped domains,
Adv. Math. 221(2009), no.5, 1725-1732.

\bibitem{GL2} P. Guan and J. Li, A mean curvature type flow in space forms,
Int. Math. Res. Not. (2015), no. 13, 4716-4740.

\bibitem{GS} M. Ghomi and J. Spruck, Total mean curvatures of Riemannian hypersurfaces,
Adv. Nonlinear Stud. 23 (2023), no. 1, Paper No. 20220029, 10 pp.

\bibitem{H} R. Hamilton, The formation of singularities in the Ricci flow,
Surveys in Differential Geom. 2 (1995), 7-136, International Press.

\bibitem{Lee} J. Lee, The spectrum of an asymptotically hyperbolic Einstein manifold,
Comm. Anal. Geom. 3 (1995), no. 1-2, 253-271.

\bibitem{L} P. Li, Geometric Analysis, Cambridge University Press, 2012.

\bibitem{LW1} P. Li and J. Wang, Complete manifolds with positive spectrum,
J. Differential Geom. 58 (2001), 501-534.

\bibitem{LW2} P. Li and J. Wang, Complete manifolds with positive spectrum,
II, J. Differential Geom. 62 (2002), 143-162.

\bibitem{LW3} P. Li and J. Wang, Weighted Poincar\'{e} inequality and
rigidity of complete manifolds, Ann. Sci. Ecole Norm. Sup. 39 (2006),
921-982.

\bibitem{McKean} H. P. McKean, 
An upper bound to the spectrum of $\Delta$ on a manifold of negative curvature,
J. Differential Geometry 4 (1970), 359-366.

\bibitem{Min} H. Minkowski, Volumen und Oberfl\"ache, Math. Ann. 57 (1903), no. 4, 447-495.

\bibitem{MW} O. Munteanu and J. Wang, Comparison theorems for 3D manifolds
with scalar curvature bound, Int. Math. Res. Not. 3 (2023), 2215-2242.

\bibitem{MW1} O. Munteanu and J. Wang, Smooth metric measure spaces with
non-negative curvature, Comm. Anal. Geom, 19 (2011), no. 3, 451-486.

\bibitem{MW2} O. Munteanu and J. Wang, Analysis of the weighted Laplacian
and applications to Ricci solitons, Comm. Anal. Geom. 20 (2012), no. 1,
55-94.

\bibitem{P} S. Patterson, The limit set of a Fuchsian group, Acta Math. 136 (1976), 241-273.


\bibitem{PRS} S. Pigoli, M. Rimoldi and A. Setti, Remarks on noncompact
non-compact gradient Ricci solitons, Math. Z.  268 (2011),  777-790.

\bibitem{S} D. Sullivan, 
Related aspects of positivity in Riemannian geometry,
J. Differential Geom. 25 (1987), no.3, 327-351.

\bibitem{Wang} X. Wang, A new proof of Lee's theorem on the spectrum of conformally 
compact Einstein manifolds, Comm. Anal. Geom. 10 (2002), no. 3, 647-651.

\bibitem{Yau} S. T. Yau, 
Harmonic functions on complete Riemannian manifolds,
Comm. Pure Appl. Math. 28 (1975), 201-228.

\end{thebibliography}
\end{document}